\documentclass[a4paper,10pt]{amsart}

\usepackage[T1]{fontenc}
\usepackage[latin1]{inputenc}
\usepackage{cite}
\usepackage{amsmath,amsthm}
\usepackage{amssymb,latexsym}
\usepackage{enumerate}
\usepackage{amsfonts}
\usepackage{amsmath}
\usepackage{amsxtra,amscd}
\usepackage{mathrsfs}
\usepackage{hyperref}
\usepackage{graphicx}
\usepackage{blindtext}

\newcommand{\R}{\mathbb{R}}

\newcommand{\N}{\mathbb{N}}

\newcommand{\K}{\mathbb{K}}

\newtheorem{theorem}{Theorem}

\newtheorem{proposition}[theorem]{Proposition}

\newtheorem{definition}[theorem]{Definition}

\title[A short remark on the surjectivity of the combinatorial Laplacian]{A short remark on the surjectivity of the combinatorial Laplacian on infinite graphs}
\author{T. Kalmes}

\begin{document}

\begin{abstract}
Applying a well-known theorem due to Eidelheit, we give a short proof of the surjectivity of the combinatorial Laplacian on connected locally finite undirected simplicial graph $G$ with countably infinite vertex set $V$ established in \cite{CeCoDo12}. In fact, we show that every linear operator on $\K^V$ which has finite hopping range and satisfies the pointwise maximum principle is surjective.
\end{abstract}

\maketitle

\section{Introduction}

In \cite{CeCoDo12} Ceccherini-Silberstein, Coornaert, and Dodziuk showed that on a connected locally finite simplicial undirected graph $G$ with a countably infinite vertex set $V$ the combinatorial Laplacian on the real valued functions on $V$, $\Delta_G:\R^V\rightarrow\R^V$ defined by
\[\forall\,v\in V:\,\Delta_G\,f(v)=f(v)-\frac{1}{\mbox{deg}(v)}\sum_{v\sim w}f(w),\]
is surjective. Two vertices $v,w\in V$ of $G$ are adjacent, $v\sim w,$ if $(v,w)$ is an edge of $G$. Recall that a graph is locally finite if for every vertex $v$ of $G$ the number $\mbox{deg}(v)$ of adjacent vertices is finite and that $G$ is connected if for any pair of different vertices $v,w$ there is a finite number of edges $(v_0,v_1),\ldots,(v_{n-1},v_n)$ of $G$ with $v\in\{v_0,v_n\}$ and $w\in\{v_0,v_n\}$. Finally, $G$ is simplicial if it does not have any loops, i.e.\ $(v,v)$ is not an edge of $G$ for any vertex $v$, and $G$ is undirected if $(w,v)$ is an edge of $G$ whenever $(v,w)$ is an edge of $G$.

In \cite{CeCoDo12} the surjectivity of the Laplacian was proved by a Mittag-Leffler argument and an application of the (pointwise) maximum principle for $\Delta_G$. There is a vast amount of literature dealing with a systematic study of the Mittag-Leffler procedure and its applications to a wide range of surjectivity problems, see e.g.\ \cite{Wengenroth03} and the references therein. As noted in \cite{CeCoDo12}, the prove of surjectivity of $\Delta_G$ extends to more general operators, e.g.\ to operators of the form $\Delta_G+\lambda$, where $\lambda:V\rightarrow [0,\infty)$.

The aim of this note is to give a very short proof of the following generalisation of the above result.\\

\textbf{Theorem 1.} \textit{Let $G$ be a locally finite connected graph (which may be directed or undirected) with countably infinite vertex set $V$. Every linear operator $A:\K^V\rightarrow\K^V$ which has finite hopping range and satisfies the pointwise maximum principle is surjective.}\\

Note that $\Delta_G$ has finite hopping range and satisfies the pointwise maximum principle. (For a precise definition of finite hopping range and the pointwise maximum principle see section 2.) The proof of the above theorem relies on a well-known result due to Eidelheit.

\section{Proof of Theorem 1}

We equip $\K^\N$ with its usual Fr\'echet space topology, i.e.\ the locally convex topology defined by the increasing fundamental system of seminorms $(p_k)_{k\in\N}$ given by
\begin{equation}\label{norms}
\forall\,f=(f_j)_{j\in\N}\in\K^\N:\,p_k(f)=\sum_{j=1}^k|f_j|.
\end{equation}
As usual, we denote this Fr\'echet space by $\omega$. The dual space $\omega'$ of $\omega$ is given by the space of finitely supported sequences $\varphi$. For $j\in\N$ we set $\pi_j:\omega\rightarrow\K,(f_m)_{m\in\N}\mapsto f_j$ so that $\varphi=\mbox{span}\{\pi_j;\,j\in\N\}$.

If $E$ is any Fr\'echet space and $A:E\rightarrow\omega$ linear and continuous we set $A_j:=\pi_j\circ A, j\in\N,$ so that $A_j\in E'$. A straight forward calculation gives that the transpose $A^t:\varphi\rightarrow E'$ is given by $A^t(y)=\sum_{j=1}^\infty y_j A_j$, where only finitely many $y_j$'s do not vanish. Thus, the linear independence of $(A_j)_{j\in\N}$ is equivalent to the injectivity of $A^t$. By the Hahn-Banach Theorem, the later is equivalent to $A$ having dense range. We recall the following theorem due to Eidelheit (see e.g.\ \cite{Eidelheit1936} or \cite[Theorem 26.27]{MeVo1997}).\\

\textbf{Eidelheit's Theorem} \textit{Let $E$ be a Fr\'echet space, $(p_k)_{k\in\N}$ be an increasing fundemantal system of seminorms on $E$ and let $A:E\rightarrow\omega$ be linear and continuous. Then, $A$ is surjective if, and only if, for $(A_j)_{j\in\N}$ defined as above the following conditions are satisfied.
\begin{itemize}
\item[i)] $(A_j)_{j\in\N}$ is linearly independent.
\item[ii)] For every $k\in\N$
\[\mbox{dim}(\{\phi\in E';\,\exists\,c>0:|\phi|\leq c\, p_k\}\cap\mbox{span}\{A_j;\,j\in\N\})<\infty.\]
\end{itemize}
}

For $E=\omega$ with the increasing fundamental system of seminorms $(p_k)_{k\in\N}$ given by $(\ref{norms})$, it is immediate that for $y\in\varphi$ there is $c>0$ with $|\langle y,f\rangle|\leq cp_k(f)$ for all $f\in\omega$ precisely when $y\in\varphi_k:=\{x=(x_j)_{j\in\N}\in\omega;\,x_j=0\mbox{ for all }j>k\}$. Since $\varphi_k$ is obviously finite dimensional, by Eidelheit's Theorem, a linear continuous operator $A:\omega\rightarrow\omega$ is surjective if and only if $(A_j)_{j\in\N}$ is linearly indepent.\\

Now, let $G$ be a locally finite connected graph with countably infinite vertex set $V$. For each $v\in V$ and $n\in\N$ we define $U_n(v)$ to be the union of $\{v\}$ with the set of all endpoints of paths in $G$ staring in $v$ and of length not exceeding $n$ together with the set of starting points of paths in $G$ ending in $v$ and of length not exceeding $n$.

\begin{definition}\label{defi}
\begin{rm}
Let $G$ be a graph with vertex set $V$ and let $A:\K^V\rightarrow\K^V$ be linear.
\begin{itemize}
\item[i)] $A$ has {\it finite hopping range} if for every $v\in V$ there is $n\in\N$ such that the following implication holds
\[\forall\,f,g\in\K^V:f_{|U_n(v)}=g_{|U_n(v)}\Rightarrow A(f)(v)=A(g)(v).\]
\item[ii)] $A$ satisfies the \textit{pointwise maximum principle} if for every $v\in V$ there is $n\in\N$ such that for each $f\in\K^V$ with $A(f)(v)=0$ the implication
\[|f(v)|=\max_{U_n(v)}|f(w)|\Rightarrow \forall\,w\in U_n(v):\,|f(w)|=|f(v)|\]
holds.
\end{itemize}
\end{rm}
\end{definition}

Enumeration of the vertices $V=\{v_k;\,k\in\N\}$ of $G$ clearly gives an isomorphism of $\K^V$ onto $\K^\N$. In order to keep notation simple, for a linear mapping $A:\K^V\rightarrow\K^V$ we denote by $A$ also the linear operator on $\K^\N$ induced by the canonical isomporphism between $\K^V$ and $\K^\N$. Since the linear operator $A:\omega\rightarrow\omega$ is continuous if and only if $A_j\in\varphi$ for all $j\in\N$, using that $G$ is connected and locally finite, a straight forward calculation shows that $A:\omega\rightarrow\omega$ is continuous if and only if the inducing $A:\K^V\rightarrow \K^V$ has finite hopping range.

\begin{proposition}
Let $G$ be a locally finite connected graph with countably infinite vertex set $V$ and let $A:\K^V\rightarrow\K^V$ be linear. If $A$ has finite hopping range and satisfies the pointwise maximum principle, then $(A_j)_{j\in\N}$ is a linearly independent sequence of continuous linear forms on $\omega$.
\end{proposition}

\begin{proof} As pointed out above, the claim is equivalent to the injectivity of $A^t$. For $k\in\N$ we define
\[M_k:\omega\rightarrow\omega,(f_j)_{j\in\N}\mapsto (f_1,\ldots,f_k,0,\ldots).\]
Then $M_k$ is a continuous linear operator on $\omega$ with $M_k^t=M_{k|\varphi}$.

Fix $y\in\varphi$ with $A^t(y)=0$. Then there is $k\in\N$ with $y\in\varphi_k$ and
\begin{eqnarray*}
\forall\,f\in\varphi_k:0&=&\langle y,A(f)\rangle=\langle M_k^t(y),(A_{|\varphi_k}(f)\rangle=\langle y,(M_k\circ A_{|\varphi_k})(f)\rangle\\
&=&\langle (M_k\circ A_{|\varphi_k})^{t_k}(y),f\rangle_{k},
\end{eqnarray*}
where $\langle\cdot,\cdot\rangle_{k}$ denotes the duality bracket between $\varphi_k'$ and $\varphi_k$ and $(M_k\circ A_{|\varphi_k})^{t_k}$ the transpose of
\[M_k\circ A_{|\varphi_k}:\varphi_k\rightarrow\varphi_k\]
with respect to this duality. From the above we conclude that $(M_k\circ A_{|\varphi_k})^{t_k}(y)=0$ so that $y=0$ if we can show that $(M_k\circ A_{|\varphi})^{t_k}$ is injective. Since $\varphi_k$ is finite dimensional, the later holds precisely when $M_k\circ A_{|\varphi_k}$ is injective, which will be proved as in \cite{CeCoDo12} by  the pointwise maximum principle for $A$.

For $j\in\N$ we define $N(j)=\{l\in\N;\,v_l\in U_n(v_j)\}$, where $n$ is chosen as in Definition \ref{defi} ii) for $v=v_j$. Let $f\in\varphi_k$ satisfy $(M_k\circ A)(f)=0$ and let $1\leq k_0\leq k$ be such that $|f_{k_0}|=\max\{|f_j|:\,j\in\N\}=:M$. By $0=(M_k\circ A)(f)(k_0)=A(f)(k_0)$ and the pointwise maximum principle we conclude $|f_j|=M$ for all $j\in N(k_0)$. If there is $j\in N(k_0)$ with $j>k$ then $M=0$ , since $f$ vanishes in $\{k+1,\ldots\}$. If $N(k_0)\subseteq\{1,\ldots,k\}$ it follows again from the maximum principle that $|f_j|=M$ for all $j\in\cup_{l\in N(k_0)}N(l)$. Using that $G$ is connected, a finite numbers of iterations of this process finally yields that $|f_j|=M$ for some $j\in\{k+1,\ldots\}$,  where $f$ vanishes so that $M=0$, i.e.\ $f=0$.
\end{proof}

\textit{Proof of Theorem 1.} Theorem 1 follows now immediately from the considerations following Eidelheit's Theorem and the above proposition.\hfill$\square$\\

\noindent\textbf{Acknowledgement.} I want to thank Daniel Lenz for turning my attention to the article \cite{CeCoDo12}.

\begin{small}
{\sc Technische Universit\"at Chemnitz,
Fakult\"at f\"ur Mathematik, 09107 Chemnitz, Germany}

{\it E-mail address: thomas.kalmes@mathematik.tu-chemnitz.de}
\end{small}
\end{document}